\documentclass[smallextended,referee,envcountsect]{svjour3}
\smartqed
\journalname{}

\usepackage{amsmath,amssymb,mathtools,bm}
\usepackage{booktabs,array,float,graphicx}
\usepackage{xcolor}
\usepackage{cite}
\usepackage[hidelinks]{hyperref}

\allowdisplaybreaks[2]
\spnewtheorem{assumption}{Assumption}[section]{\bfseries}{\rmfamily}

\usepackage{algorithm}
\usepackage[noEnd=false,indLines=true]{algpseudocodex}
\usepackage{multirow}
\DeclareMathOperator*{\argmin}{arg\,min}
\DeclareMathOperator{\prox}{prox}
\DeclareMathOperator{\soft}{soft}
\DeclareMathOperator{\Fix}{Fix}
\DeclareMathOperator{\supp}{supp}
\DeclareMathOperator{\dist}{dist}
\DeclareMathOperator{\dom}{dom}

\newcommand{\R}{\mathbb R}
\newcommand{\D}{\mathcal D}

\newcommand{\ip}[2]{\left\langle #1,#2\right\rangle}
\newcommand{\norm}[1]{\left\lVert #1\right\rVert}

\begin{document}

\title{Restarted Fixed-Point Acceleration of Proximal ADMM for Capped-$\ell_1$ Penalized Optimization}
\titlerunning{Restarted Fixed-Point Acceleration of Proximal ADMM}

\author{Hongyuan Tao \and Fan Wu \and Wei Bian}

\institute{
Hongyuan Tao \at
School of Mathematics, Harbin Institute of Technology \\
Harbin, China \\
taohy163@163.com
\and
Fan Wu \at
School of Mathematics, Harbin Institute of Technology \\
Harbin, China \\
wufanmath@163.com
\and
Wei Bian, Corresponding author \at
School of Mathematics, Harbin Institute of Technology \\
Harbin, China \\
bianweilvse520@163.com
}
\date{Received: date / Accepted: date}
\maketitle

\begin{abstract}
We consider box-constrained optimization with a convex, Lipschitz continuous loss function that may be nonsmooth and a capped-$\ell_1$ penalty. Exploiting the convex branch
representation of the capped-$\ell_1$ penalty, we construct an
associated nonexpansive fixed-point mapping for each branch and,
based on these mappings, propose a restarted accelerated proximal
alternating direction method of multipliers (RA-pADMM). The resulting algorithm consists of a preliminary
branch identification phase and a subsequent accelerated phase
based on either the Halpern or fast
Krasnosel'ski\u{\i}--Mann iteration, together with a
sufficient decrease restart mechanism. We prove that the preliminary phase terminates after finitely many
iterations and that only finitely many branch updates are accepted
during the subsequent phase. The
primal iterates and proximal outputs converge to a common lifted
stationary point, which is also a $\nu$-strong local minimizer of
the corresponding cardinality penalized problem. Moreover, the
Halpern and fast Krasnosel'ski\u{\i}--Mann variants achieve
nonergodic $O(k^{-1})$ and asymptotic $o(k^{-1})$ convergence
rates, respectively, for both the objective value gap and the
Karush--Kuhn--Tucker residual associated with the final
branch. Numerical experiments on sparse recovery and real data
classification problems demonstrate the effectiveness of the proposed
RA-pADMM algorithm.
\end{abstract}

\keywords{capped-$\ell_1$ penalty \and proximal ADMM \and Halpern iteration \and fast Krasnosel'ski\u{\i}--Mann iteration}
\subclass{90C26 \and 90C30 \and 65K05}

\section{Introduction}\label{sec:introduction}

In this paper, we consider the box-constrained cardinality penalized problem
\begin{equation}\label{model_l0}
    \min_{x\in X}\;F_{\ell_0}(x):=f(Ax)+\lambda\|x\|_0,
\end{equation}
where $A\in\mathbb{R}^{m\times n}$, $\lambda>0$,
$X=\prod_{i=1}^n[l_i,u_i]$ is a nonempty box set, and
$f:\mathbb{R}^m\to\mathbb{R}$ is convex and possibly nonsmooth.
The $\ell_0$ penalty is a natural and widely used choice for promoting sparsity \cite{Donoho2006,BrucksteinDonohoElad2009}, as it directly penalizes the number of nonzero components, while the loss function $f(Ax)$ measures the data-fitting quality. However, even when $f$ is convex, problem~\eqref{model_l0} remains nonconvex and combinatorial due to the discontinuous cardinality term, and is generally difficult to solve globally \cite{Natarajan1995}.

To retain the sparsity-inducing structure of the cardinality penalty
while avoiding its discontinuity, a common strategy is to replace the
$\ell_0$ term with a continuous nonconvex surrogate
\cite{FanLi2001,LvFan2009,Zhang2010,LeThiEtAl2015,
GassoEtAl2009,SoubiesEtAl2015,SoubiesEtAl2017,AhnPangXin2017}. Essafri et al.~\cite{EssafriCalatroniSoubies2025} recently
developed exact continuous relaxations based on Bregman distances
for $\ell_0$-regularized problems with non-quadratic data terms.
Among these continuous relaxations, we focus on the capped-$\ell_1$ penalty, which leads to the following relaxation model:
\begin{equation}\label{model}
    \min_{x\in X}\;F(x):=f(Ax)+\lambda\Phi(x),
    \qquad
    \Phi(x):=\sum_{i=1}^n
    \min\left\{\frac{|x_i|}{\nu},1\right\},
\end{equation}
where $\nu>0$ is a cap parameter.
The capped-$\ell_1$ penalty has a simple piecewise-linear structure, growing linearly near zero and remaining constant once the magnitude of the variable reaches $\nu$. Thus, it preserves the sparsity-inducing effect of the cardinality penalty while alleviating its discontinuity.
Under suitable conditions, problem~\eqref{model} is an exact continuous relaxation of problem~\eqref{model_l0} with the same set of global minimizers, and every lifted stationary point of problem~\eqref{model} is a $\nu$-strong local minimizer of problem~\eqref{model_l0}~\cite{bian2020}.
These properties provide a theoretical foundation for studying problem~\eqref{model} as a tractable continuous reformulation of problem~\eqref{model_l0}.

Our main concern in this paper is to solve problem~\eqref{model} more efficiently.
Bian and Chen~\cite{bian2020} proposed a smoothing proximal gradient
method, and under their stated assumptions, established an
objective value convergence rate of $o(k^{-\sigma})$ for any
$\sigma\in(0,1/2)$, together with convergence of the whole sequence
to a lifted stationary point of problem~\eqref{model}.
More recently, Tao et al.~\cite{TaoBianQu2026} developed an
accelerated nested proximal-gradient method based on Moreau-envelope
smoothing and selective extrapolation. Their accelerated scheme attains
an $O(1/k)$ objective value convergence rate but guarantees only
subsequential convergence to lifted stationary points. They also
proposed a modified extrapolation scheme that ensures convergence of
the entire sequence, possibly with a weaker theoretical rate.
These results raise the question of whether an $O(1/k)$ rate can be
achieved while preserving convergence of the entire sequence.

To exploit the separable structure of problem~\eqref{model}, we
introduce an auxiliary variable $y=Ax$. The resulting formulation allows us to apply either the alternating
direction method of multipliers (ADMM) or Douglas--Rachford
splitting (DRS) to solve it
\cite{LionsMercier1979,EcksteinBertsekas1992,BoydEtAl2011}.
For convex problems, both methods admit fixed-point interpretations
\cite{Rockafellar1976,EcksteinBertsekas1992,BauschkeCombettes2017},
which enables the use of accelerated fixed-point iterations.
In particular, for suitable nonexpansive mappings, the Halpern
iteration~\cite{Halpern1967} attains an $O(1/k)$ fixed-point
residual rate~\cite{Lieder2021}, whereas the fast
Krasnosel'ski\u{\i}--Mann (fast KM) iteration achieves the asymptotic
rate $o(1/k)$~\cite{BotNguyen2023}.
These residual estimates yield convergence rates for the
objective function value gap and the
Karush--Kuhn--Tucker (KKT) residual~\cite{SunYuanZhangZhao2025}. 
For ADMM-type methods applied to convex problems, $O(1/k)$ convergence rates for the
objective value gap and primal feasibility are commonly established
in an ergodic sense \cite{HeYuan2012,TangToh2024}. Recently, Sun et al.~\cite{SunYuanZhangZhao2025} reformulated preconditioned
ADMM as a degenerate proximal-point mapping
\cite{BrediesSun2017,BrediesEtAl2022} and incorporated Halpern and fast KM iterations into this framework. For the convex problem considered therein, the two resulting schemes attain, respectively, a nonasymptotic $O(1/k)$ rate and an asymptotic $o(1/k)$ rate for both the KKT residual and the primal objective value gap. Moreover, the iterates generated by both schemes are convergent. 
These results, which combine global convergence of the sequence with nonergodic convergence rates, motivate us to explore the use of the Halpern and fast KM acceleration mechanisms in the setting considered here.
The Halpern iteration has also been used to accelerate
operator splitting methods in various applications. Zhang et al.~\cite{ZhangZhangWang2025}
proposed a Halpern-type preconditioned proximal point algorithm
and applied it to image restoration.
Chen et al.~\cite{ChenEtAl2026HPRLP} developed HPR-LP, a linear
programming solver based on the Halpern Peaceman--Rachford (HPR)
method with adaptive restarts and penalty parameter updates.


One main challenge is that the capped-$\ell_1$ penalty makes
problem~\eqref{model} nonconvex, so the fixed-point reformulation
of ADMM used in \cite{SunYuanZhangZhao2025} to analyze convergence
for convex problems is not directly applicable. To address this issue, we exploit the difference-of-convex (DC) representation of the capped-$\ell_1$ penalty used in \cite{bian2020}. Then, the problem \eqref{model} can be equivalently written in the following DC form:
\begin{equation}\label{eq:dc-decomposition}
 \min_{x\in\mathbb{R}^n} \mathcal F(x)
  := f(Ax)+ \delta_X(x)
  + \frac{\lambda}{\nu}\|x\|_1
  - \lambda\sum_{i=1}^n
  \max_{r\in\{1,2,3\}}\theta_r(x_i),
\end{equation}
where $\theta_1(t)=0$, $\theta_2(t)=\frac{t}{\nu}-1$,
and $\theta_3(t)=-\frac{t}{\nu}-1$;
$\delta_X$ denotes the indicator function of $X$.
We note that the branch function 
$$\Phi^d(x):=\frac{1}{\nu}\|x\|_1-\sum_{i=1}^n\theta_{d_i}(x_i)$$ 
is convex for every fixed branch $d\in\D:=\{1,2,3\}^n$. Thus, instead of applying a fixed-point acceleration scheme directly to the proximal ADMM (pADMM) with the nonconvex penalty $\Phi$, we fix a branch $d$ at each iteration and work with the corresponding branch function $\Phi^d$. For each such branch, we construct an associated nonexpansive fixed-point mapping and incorporate the fixed-point acceleration mechanism into the resulting convex subproblem. This leads to a restarted accelerated proximal alternating direction method of multipliers (RA-pADMM) for solving problem~\eqref{model}. The main contributions of this paper are summarized as follows.
\begin{itemize}
    \item We reformulate the pADMM update for each convex
  branch as a fixed-point iteration.
  We further establish the nonexpansiveness of both the branch
  mapping and its relaxed counterpart with respect to a suitable weighted norm.

    \item  We propose the RA-pADMM algorithm, which incorporates Halpern
    and fast KM acceleration
    into a two-phase framework with a sufficient decrease
    restart mechanism. We prove that the preliminary phase terminates after finitely many
iterations and that only finitely many branch updates are accepted
during the subsequent phase.

 \item We prove that the primal iterates and proximal outputs
  converge to a common lifted stationary point of
  problem~\eqref{model}, which is also a $\nu$-strong local
  minimizer of problem~\eqref{model_l0}. For the Halpern and
    fast KM variants, we establish nonergodic
    $O(k^{-1})$ and asymptotic $o(k^{-1})$ rates,
    respectively, for the objective value gap and
    Karush--Kuhn--Tucker (KKT) residual associated
    with the final branch. Numerical experiments
    demonstrate the effectiveness of the proposed
    algorithm.
\end{itemize}

The remainder of the paper is organized as follows.
Section~\ref{sec:preliminaries} develops the branchwise pADMM formulation, constructs its fixed-point representation, and establishes the nonexpansiveness of the resulting operator.
Section~\ref{sec:algorithm} presents the RA-pADMM algorithm and establishes finite termination of the
preliminary identification phase, eventual branch consistency,
convergence of the entire sequence, and the corresponding convergence
rates. Section~\ref{sec:numerical} reports the numerical experiments,
and Section~\ref{sec:conclusion} concludes the paper.

\section{Branchwise pADMM and its fixed-point representation}
\label{sec:preliminaries}

We denote the set of nonnegative integers by
$\mathbb N:=\{0,1,2,\ldots\}$. We use $\langle\cdot,\cdot\rangle$ and $\|\cdot\|$ for the Euclidean inner product and norm, respectively, $\|\cdot\|_1$ for the $\ell_1$-norm, and $\|A\|$ for the induced operator norm of a matrix $A$. Let $I_{n}$ denote the $n\times n$ identity matrix. For a symmetric positive semidefinite matrix $M$, we define
$\|z\|_M:=\sqrt{\langle z,Mz\rangle}$; when $M\succ0$, let
$\langle u,v\rangle_M:=\langle u,Mv\rangle$. For a proper closed convex function $h$, $\dom h$, $\partial h$, $h^*$, and $\prox_{\gamma h}$ denote its effective domain, convex subdifferential, Fenchel conjugate, and proximal mapping, respectively, where $\prox_{\gamma h}(v):=\argmin_z\{\gamma h(z)+\frac12\|z-v\|^2\}$. For a nonempty closed convex set $C$, $\Pi_C$ denotes
its Euclidean projector, and $\dist(z,C)$ denotes
the distance from $z$ to $C$. For a mapping $T$, let $\Fix T:=\{z:Tz=z\}$. Let $\lceil t\rceil$ denote the
smallest integer greater than or equal to $t$. Define $\mathcal D^x:=\{d\in\D:d_i\in\arg\max_{r\in\{1,2,3\}}\theta_r(x_i),\ i=1,\ldots,n\}$. 
We define $(t)_+:=\max\{t,0\}$, $\supp(x):=\{i:x_i\neq0\}$ and 
\begin{equation}\label{new}
\omega := \min\{ |l_i|, |u_j| : l_i \neq 0, u_j \neq 0, i = 1,2,\ldots,n, j = 1,2,\ldots,n \}.
\end{equation}

With the notation introduced above, we next provide two solution concepts used throughout the paper.

\begin{definition}[{\cite[Definition 2.1]{bian2020}}]\label{def:lifted-stationary}
A point $x\in X$ is called a \emph{lifted stationary point} of problem
\eqref{model} if there exists $d\in\mathcal D^x$ such that
\begin{equation}\label{eq:lifted-stationarity}
  \lambda\sum_{i=1}^n\theta_{d_i}'(x_i)\mathbf e_i
  \in
  A^\top\partial f(Ax)
  +\frac{\lambda}{\nu}\partial\|x\|_1
  +N_X(x),
\end{equation}
where $N_X(x)$ denotes the normal cone to $X$ at $x$, and
$\mathbf e_i$ denotes the $i$th column of the identity
matrix $I_n$.
\end{definition}

For simplicity, we define 
\begin{equation}\label{eq:gap-set}
  \mathcal G_\nu:=
  \{x\in X:x_i=0\ \text{or}\ |x_i|\geq\nu,\ i=1,\ldots,n\}.
\end{equation}

\begin{definition}[{\cite[Definition 2.6]{bian2020}}]\label{def:strong-local-minimizer}
A point $\tilde x\in X$ is called a $\nu$-strong local
minimizer of problem~\eqref{model_l0} if $\tilde x\in\mathcal G_\nu$ and it is a local
minimizer of \eqref{model_l0}.
\end{definition}

For any branch $d\in\mathcal{D}$, we define 
\begin{equation}\label{eq:branch}
  g_d(x):=\lambda\Phi^d(x)+\delta_X(x)
\end{equation}
and
\[
   F_d:=f\circ A+\lambda\Phi^d,
  \qquad
  \mathcal F_d:=f\circ A+g_d.
\]
By the definition of $\Phi$, we have $\Phi^d(x)\geq\Phi(x)$ for every $d\in\D$, and $\Phi^d(x)=\Phi(x)$ for every
$d\in\mathcal D^x$. 
Unless otherwise stated, the following assumptions are maintained throughout the remainder of this paper.

\begin{assumption}\label{ass1}
The function $f:\R^m\to\R$ is convex and $L_f$-Lipschitz
continuous on $\R^m$ for some $L_f>0$.
\end{assumption}

\begin{assumption}\label{ass2}
The cap parameter $\nu$ satisfies $0<\nu<\min\left\{\omega, \frac{\lambda}{L_f\norm{A}}\right\}$.
\end{assumption}

\begin{assumption}\label{ass3}
For every $d\in\D$,
$\argmin_{x\in\mathbb R^n}\mathcal F_d(x)\neq\varnothing$.
\end{assumption}

\begin{remark}\label{rem:branch-solvability}
Assumption~\ref{ass3} is satisfied if $X$ is bounded or $f$ is level-bounded on $A(X)$, where $A(X):=\{Ax:x\in X\}$.
Indeed, for each $d\in\mathcal D$, define
$\psi_d(u):=\inf\{g_d(x):Ax=u\}$. Since $g_d$ is a proper closed polyhedral convex function
that is bounded below by $0$, $\psi_d$ is also a proper closed
polyhedral convex function
\cite[Corollary~19.3.1]{Rockafellar1970},
and then the infimum defining $\psi_d(u)$ is attained
for every $u\in A(X)$. Moreover,
$\psi_d\geq0$ and
$\inf_x\mathcal F_d(x)=\inf_u\{f(u)+\psi_d(u)\}$. When $f$ is level-bounded on $A(X)$, the function $f+\psi_d$ is proper, lower semicontinuous, and level-bounded, and hence attains its minimum. Since $\mathcal D$ is finite and
$\mathcal F(x)=\min_{d\in\mathcal D}\mathcal F_d(x)$,
Assumption~\ref{ass3} also ensures that 
$\mathcal F_{\inf}:=\min_{x\in\mathbb R^n}\mathcal F(x)>-\infty$.
\end{remark}

The following lemma shows that the branch index associated with a lifted
stationary point of problem~\eqref{model} is unique. This motivates the choice of the branch indices in the algorithm
developed in the next section.

\begin{lemma}[{\cite[Proposition 2.2, Lemma 2.3]{bian2020}}]\label{lmdk}
If $x^*$ is a lifted stationary point of problem~\eqref{model},
then $x^*\in\mathcal G_\nu$, where $\mathcal G_\nu$ is defined as in \eqref{eq:gap-set}. 
Moreover, the unique element $d^{x^*}\in\mathcal D^{x^*}$ for which \eqref{eq:lifted-stationarity} holds is given by
\begin{equation}\label{lmdk-1}
d^{x^*}_i=
\begin{cases}
1, & |x^*_i|<\nu,\\
2, & x^*_i\geq\nu,\\
3, & x^*_i\leq-\nu,
\end{cases}
\qquad i=1,\ldots,n.
\end{equation}
\end{lemma}

Throughout the remainder of this paper, for any $x\in X$,
we define $d^x$ by the same componentwise rule
in \eqref{lmdk-1}. Then $d^x\in\mathcal D^x$.

For a fixed $d\in\mathcal D$, the convex branch problem $\min_{x}\mathcal{F}_{d}(x)$ admits the
equivalent splitting formulation
\begin{equation}\label{eq:convex-branch-problem}
  \min_{x,y}\; g_d(x)+f(y)
  \qquad\text{subject to}\qquad Ax-y=0.
\end{equation}
A triple $(x,y,p)\in\mathbb{R}^n\times \mathbb{R}^m\times \mathbb{R}^m$ is said to satisfy the KKT system for problem \eqref{eq:convex-branch-problem} if
\begin{equation}\label{eq:branch-kkt}
  0\in\partial g_d(x)+A^\top p,
  \qquad
  p\in\partial f(y),
  \qquad
  Ax-y=0.
\end{equation}
The following remark shows that a lifted stationary point of problem \eqref{model}
can be obtained by solving the above convex branch problem.

\begin{remark}\label{eq_lifed}
A point $x^*\in X$ is a lifted stationary point of problem \eqref{model}
if and only if there exist $d^*\in\mathcal D^{x^*}$ and $p^*\in\R^m$
such that
\begin{equation}\label{eq:lifted-kkt}
  0\in\partial g_{d^*}(x^*)+A^\top p^*,
  \qquad
  p^*\in\partial f(Ax^*).
\end{equation}
Thus, if $(x^*,Ax^*,p^*)$ satisfies the KKT
system~\eqref{eq:branch-kkt} for some
$d^*\in\mathcal D^{x^*}$, then $x^*$ is a lifted stationary
point of problem~\eqref{model}.
\end{remark}

To solve \eqref{eq:convex-branch-problem}, we define $P:=\sigma(\eta I_n -A^\top A)$ and the augmented Lagrangian
\begin{equation}\label{eq:lagrangian}
  \mathcal L_\sigma^d(x,y,p)
  :=g_d(x)+f(y)+\ip{p}{Ax-y}
  +\frac{\sigma}{2}\norm{Ax-y}^2,
\end{equation}
where $\sigma>0$ and $\eta>\norm{A}^2$. The choice of $\eta$ ensures that \(P\succ0.\)
Given an input $w=(x,p)$, one pADMM call for branch $d$
takes the form
\begin{equation}\label{eq:padmm-subproblems}
\left\{
\begin{aligned}
  \bar y
  &=\arg\min_y\left\{
  f(y)-\ip{p}{y}+\frac{\sigma}{2}\norm{Ax-y}^2
  \right\},\\
  \bar x
  &=\arg\min_{x'}\left\{
  \mathcal L_\sigma^d(x',\bar y,p)
  +\frac12\norm{x'-x}_P^2
  \right\},\\
  \bar p
  &=p+\sigma(A\bar x-\bar y).
\end{aligned}
\right.
\end{equation}

Let $\mathcal W:=\R^n\times\R^m$ and
$\tau:=1/(\sigma\eta)$. For each $d\in\D$,
define the branch fixed-point mapping
$S_d:\mathcal W\to\mathcal W$ at
$w=(x,p)\in\mathcal W$ by
\begin{equation}\label{eq:S-map}
  S_d(w)
  :=
  \binom{
    \prox_{\tau g_d}
    \bigl(x-\tau A^\top\xi(w)\bigr)
  }{
    \xi(w)+\sigma A\bigl(
      \prox_{\tau g_d}
      \bigl(x-\tau A^\top\xi(w)\bigr)-x
    \bigr)
  },
\end{equation}
where $\xi(w)
  :=
  \prox_{\sigma f^*}(p+\sigma Ax)$.

The following proposition shows that the iteration defined by
$S_d$ is equivalent to the pADMM iteration for branch $d$.

\begin{proposition}\label{prop:padmm-fixed-point}
For any $d\in\D$ and $w=(x,p)\in\R^{n+m}$, let
$(\bar y,\bar x,\bar p)$ be generated by
\eqref{eq:padmm-subproblems}. Then
\(S_d(w)=(\bar x,\bar p)\), and
\begin{equation}\label{eq:fix-Sd}
\Fix S_d
=
\left\{
(x,p):
0\in\partial g_d(x)+A^\top p,\ 
p\in\partial f(Ax)
\right\}.
\end{equation}
Moreover, if $w=(x,p)\in\Fix S_d$, then
$(x,Ax,p)$ satisfies the KKT system
\eqref{eq:branch-kkt} for
problem~\eqref{eq:convex-branch-problem}.
\end{proposition}

\begin{proof}
The pADMM update \eqref{eq:padmm-subproblems} can be
equivalently written as
\begin{equation}\label{eq:padmm-explicit}
\left\{
\begin{aligned}
  \bar y
  &=\prox_{f/\sigma}
  \left(Ax+\frac{p}{\sigma}\right),\\
  \xi
  &=p+\sigma(Ax-\bar y),\\
  \bar x
  &=\prox_{\tau g_d}(x-\tau A^\top\xi),\\
  \bar p
  &=\xi+\sigma A(\bar x-x).
\end{aligned}
\right.
\end{equation}
By Moreau's decomposition
\cite[Theorem~14.3(ii)]{BauschkeCombettes2017}, we have
\[
\begin{aligned}
  \xi
  &=p+\sigma Ax
  -\sigma\prox_{f/\sigma}
  \left(Ax+\frac{p}{\sigma}\right)\\
  &=\prox_{\sigma f^*}(p+\sigma Ax)
  =\xi(w).
\end{aligned}
\]
Hence, by the last two equalities in
\eqref{eq:padmm-explicit} and the definition of $S_d$ in
\eqref{eq:S-map}, we obtain
$S_d(w)=(\bar x,\bar p)$. Consequently, $w=(x,p)\in\Fix S_d$ if and only if
$x=\bar x$ and $p=\bar p$. Using
\eqref{eq:padmm-explicit}, we further obtain
\[
w=(x,p)\in\Fix S_d
\quad\Longleftrightarrow\quad
\left\{
\begin{aligned}
  x
  &=\prox_{\tau g_d}(x-\tau A^\top p),\\
  Ax
  &=\prox_{f/\sigma}
  \left(Ax+\frac{p}{\sigma}\right).
\end{aligned}
\right.
\]
By the first-order optimality conditions, the two equalities on
the right-hand side are equivalent to
\[
\left\{
\begin{aligned}
  0&\in\partial g_d(x)+A^\top p,\\
  p&\in\partial f(Ax).
\end{aligned}
\right.
\]
This proves \eqref{eq:fix-Sd}. Hence, if
$w=(x,p)\in\Fix S_d$, then $(x,Ax,p)$ satisfies the KKT system
\eqref{eq:branch-kkt} of
problem~\eqref{eq:convex-branch-problem}, which completes the
proof.
\qed

\end{proof}

Assumption~\ref{ass3} ensures that
$\Fix S_d\neq\varnothing$ for every $d\in\D$.
Moreover, the proof of
Proposition~\ref{prop:padmm-fixed-point}
shows that, for $w=(x,p)\in\mathcal W$,
the $\bar x$-update in
$\bar w=(\bar x,\bar p)=S_d(w)$
coincides with that in \eqref{eq:padmm-explicit},
and $\xi(w)$ coincides with $\xi$.
Hence, we have the following remark.

\begin{remark}\label{rem:x-explicit}
Let $d\in\D$, $w=(x,p)\in\mathcal W$,
and $\bar w=(\bar x,\bar p)=S_d(w)$.
Define $s(d)\in\R^n$ componentwise by
$[s(d)]_i=0,1,-1$ when $d_i=1,2,3$,
respectively. Then the $\bar x$-update
in $\bar w=S_d(w)$ admits the
componentwise closed-form representation
\begin{equation}\label{eq:sparse-update}
  \bar x_i
  =
  \Pi_{[l_i,u_i]}\left[
  \soft\left(
  q_i+\frac{\tau\lambda}{\nu}[s(d)]_i,
  \frac{\tau\lambda}{\nu}
  \right)
  \right],
  \qquad i=1,\ldots,n,
\end{equation}
where $q=x-\tau A^\top\xi$, $\xi$ is defined in
\eqref{eq:padmm-explicit}, and
$\soft(t,a):=\operatorname{sign}(t)(|t|-a)_+$.
\end{remark}

Define
\begin{equation}\label{eq:H}
  H:=
  \begin{pmatrix}
    \tau^{-1}I_n&0\\
    0&\sigma^{-1}I_m
  \end{pmatrix}.
\end{equation}
The matrix $H$ is positive definite and hence induces the inner
product $\langle\cdot,\cdot\rangle_H$ and the associated norm
$\|\cdot\|_H$ on $\mathcal W$.
We next establish the nonexpansiveness of the branch map $S_d$ in the $H$-induced norm.

\begin{lemma}\label{lem:S-nonexpansive}
Suppose that $\eta>\|A\|^2$. Then, for every $d\in\D$, the mapping
$S_d:\mathcal W\to\mathcal W$ is nonexpansive with respect to
the norm induced by $H$, namely,
\begin{equation}\label{eq:S-nonexpansive-new}
  \|S_d(w_1)-S_d(w_2)\|_H
  \leq
  \|w_1-w_2\|_H,
  \qquad
  \forall\,w_1,w_2\in\mathcal W.
\end{equation}
\end{lemma}

\begin{proof}
For $w_i=(x_i,p_i)\in\mathcal W$, $i\in\{1,2\}$, denote $\xi_i:=\xi(w_i)$, $v_i:=
\prox_{\tau g_d}
\bigl(x_i-\tau A^\top\xi_i\bigr)$, 
$a_i:=x_i-v_i$, and $b_i:=p_i-\xi_i$.
The optimality conditions for the two proximal mappings in
\eqref{eq:S-map} give
\begin{equation}\label{eq:prox-optimality-i}
  \frac{1}{\tau}a_i-A^\top\xi_i\in\partial g_d(v_i),
  \qquad
  \frac{1}{\sigma}b_i+Ax_i\in\partial f^*(\xi_i),
  \qquad i\in\{1,2\}.
\end{equation}
Then, by the monotonicity of $\partial g_d$ and $\partial f^*$
\cite[Theorem~20.25]{BauschkeCombettes2017}, we obtain
\begin{equation}\label{eq:monotone-new}
0\leq
\left\langle v_1-v_2,
\frac{a_1-a_2}{\tau}-A^\top(\xi_1-\xi_2)\right\rangle
+
\left\langle \xi_1-\xi_2,
\frac{b_1-b_2}{\sigma}+A(x_1-x_2)\right\rangle.
\end{equation}
Using the definition of $S_d$ in \eqref{eq:S-map}, together with
$x_i=v_i+a_i$ and $p_i=\xi_i+b_i$ for $i=1,2$,
we expand the squared norms to obtain
\begin{align}
&\|w_1-w_2\|_H^2
-\|S_d(w_1)-S_d(w_2)\|_H^2 \notag\\
&=
2\left\langle v_1-v_2,
\frac{a_1-a_2}{\tau}-A^\top(\xi_1-\xi_2)\right\rangle
+
2\left\langle \xi_1-\xi_2,
\frac{b_1-b_2}{\sigma}+A(x_1-x_2)\right\rangle \notag\\
&\quad+
\frac{1}{\tau}\|a_1-a_2\|^2
-\sigma\|A(a_1-a_2)\|^2
+\frac{1}{\sigma}\|b_1-b_2\|^2 \notag\\
&\geq
\frac{1}{\tau}\|a_1-a_2\|^2
-\sigma\|A(a_1-a_2)\|^2
+\frac{1}{\sigma}\|b_1-b_2\|^2 \notag\\
&\geq
\left(\frac{1}{\tau}-\sigma\|A\|^2\right)
\|a_1-a_2\|^2
+\frac{1}{\sigma}\|b_1-b_2\|^2
\geq0,
\label{S_d_ne}
\end{align}
where the first inequality follows from
\eqref{eq:monotone-new}, and the last follows from
$1/\tau=\sigma\eta$ and $\eta>\|A\|^2$.
Therefore, \eqref{eq:S-nonexpansive-new} holds. \qed
\end{proof}

For each $d\in\D$ and $\rho\in(0,2)$, define the relaxed branch
map
\begin{equation}\label{eq:T-map}
  T_{d,\rho}:=I+\rho(S_d-I),
\end{equation}
where $S_d$ is defined as in \eqref{eq:S-map}.

\begin{lemma}\label{T-nonexpansive}
Suppose that $\eta>\|A\|^2/\min\{1,2-\rho\}^2$.
Then, for every $d\in\D$, the mapping $T_{d,\rho}$ is
nonexpansive with respect to the norm induced by $H$, namely,
\begin{equation}\label{eq:T-nonexpansive}
  \|T_{d,\rho}w_1-T_{d,\rho}w_2\|_H
  \leq
  \|w_1-w_2\|_H,
  \qquad
  \forall\,w_1,w_2\in\mathcal W.
\end{equation}
\end{lemma}

\begin{proof}
For $w_i=(x_i,p_i)\in\mathcal W$, $i\in\{1,2\}$, we use the
notation $a_i$ and $b_i$ introduced in the
proof of Lemma~\ref{lem:S-nonexpansive}. Since
$1/\tau=\sigma\eta$, we have
\begin{align}
&2\left|
\left\langle
b_1-b_2+\sigma A(a_1-a_2),
A(a_1-a_2)
\right\rangle
\right| \notag\\
&\quad\leq
\frac{\|A\|}{\sqrt{\eta}}
\left[
\frac{1}{\tau}\|a_1-a_2\|^2
+\frac{1}{\sigma}
\|b_1-b_2+\sigma A(a_1-a_2)\|^2
\right],
\label{eq:basic-lower-bound}
\end{align}
which means that
\begin{align}
&\frac{1}{\tau}\|a_1-a_2\|^2
-\sigma\|A(a_1-a_2)\|^2
+\frac{1}{\sigma}\|b_1-b_2\|^2 \notag\\
&=
\frac{1}{\tau}\|a_1-a_2\|^2 +
\frac{1}{\sigma}
\|b_1-b_2+\sigma A(a_1-a_2)\|^2
 \notag\\
&\quad -2\left\langle
b_1-b_2+\sigma A(a_1-a_2),
A(a_1-a_2)
\right\rangle \notag\\
&\geq
\left(1-\frac{\|A\|}{\sqrt{\eta}}\right)
\left[
\frac{1}{\tau}\|a_1-a_2\|^2
+\frac{1}{\sigma}
\|b_1-b_2+\sigma A(a_1-a_2)\|^2
\right] \notag\\
&=
\left(1-\frac{\|A\|}{\sqrt{\eta}}\right)
\left\|
(I-S_d)w_1-(I-S_d)w_2
\right\|_H^2.
\label{eq:residual-lower-bound}
\end{align}

The estimate established in \eqref{S_d_ne} can be rearranged as
\begin{align}
\|S_d(w_1)-S_d(w_2)\|_H^2
&\leq
\|w_1-w_2\|_H^2
-\frac{1}{\tau}\|a_1-a_2\|^2 \notag\\
&\quad+
\sigma\|A(a_1-a_2)\|^2
-\frac{1}{\sigma}\|b_1-b_2\|^2.
\label{eq:S-strengthened}
\end{align}
Combining \eqref{eq:residual-lower-bound} and
\eqref{eq:S-strengthened}, we obtain
\begin{align}
\|S_d(w_1)-S_d(w_2)\|_H^2
&\leq
\|w_1-w_2\|_H^2 \notag\\
&\quad-
\left(1-\frac{\|A\|}{\sqrt{\eta}}\right)
\left\|
(I-S_d)w_1-(I-S_d)w_2
\right\|_H^2.
\label{eq:Sd-averaged-estimate}
\end{align}

Since
\(
  S_d(w_1)-S_d(w_2)
  =
  w_1-w_2-
  \bigl[(I-S_d)w_1-(I-S_d)w_2\bigr]
\), inequality \eqref{eq:Sd-averaged-estimate} implies
\begin{align}
&2\left\langle
w_1-w_2,
(I-S_d)w_1-(I-S_d)w_2
\right\rangle_H \notag\\
&\quad\geq
\left(2-\frac{\|A\|}{\sqrt{\eta}}\right)
\left\|
(I-S_d)w_1-(I-S_d)w_2
\right\|_H^2.
\label{eq:residual-inner-product}
\end{align}
Moreover, by
\(
T_{d,\rho}(w_1)-T_{d,\rho}(w_2)
=
w_1-w_2-
\rho\bigl[(I-S_d)w_1-(I-S_d)w_2\bigr],
\)
we have
\begin{align}
&\|T_{d,\rho}(w_1)-T_{d,\rho}(w_2)\|_H^2 \notag\\
&=
\|w_1-w_2\|_H^2
-2\rho
\left\langle
w_1-w_2,
(I-S_d)w_1-(I-S_d)w_2
\right\rangle_H \notag\\
&\quad+
\rho^2
\left\|
(I-S_d)w_1-(I-S_d)w_2
\right\|_H^2 \notag\\
&\leq
\|w_1-w_2\|_H^2
-\rho
\left(
2-\rho-\frac{\|A\|}{\sqrt{\eta}}
\right)
\left\|
(I-S_d)w_1-(I-S_d)w_2
\right\|_H^2,
\label{eq:T-nonexpansive-1}
\end{align}
where the inequality follows from \eqref{eq:residual-inner-product}.
The condition on $\eta$ implies that
\(
\|A\|/\sqrt{\eta}
  <
  \min\{1,2-\rho\},
\)
and hence
\(2-\rho-\|A\|/\sqrt{\eta}>0\).
Thus, \eqref{eq:T-nonexpansive} follows from
\eqref{eq:T-nonexpansive-1}.\qed
\end{proof}

\section{Restarted acceleration of pADMM}
\label{sec:algorithm}

In this section, we propose a restarted accelerated proximal
alternating direction method of multipliers, named RA-pADMM,
for finding a lifted stationary point of problem~\eqref{model}. We prove the
convergence of the iterates and derive the convergence
rates for the KKT residual and objective value gap.

\subsection{Restarted accelerated pADMM}

The RA-pADMM algorithm consists of two phases and is detailed
in Algorithm~\ref{alg:main}. In Phase~I, the update $\bar w^j=S_{d^j}(w^j)$,
equivalent to one pADMM step for branch $d^j$, is combined
with the Halpern scheme until an output
$\bar x^j\in\mathcal G_\nu$ is obtained, thereby providing
a preliminary identification of the branch structure. In Phase~II, accelerated iterations are performed on a branch
that remains fixed within each epoch. A restart is triggered
when the hard-thresholded candidate induces a different branch
and yields a sufficient decrease in the original objective function.

To describe the restart mechanism, define the hard-thresholding map
\begin{equation}\label{eq:Z}
  (Z_\nu(x))_i:=
  \begin{cases}
    0, & |x_i|<\nu,\\
    x_i, & |x_i|\geq\nu.
  \end{cases}
\end{equation}
For every $v\in X$, Assumption~\ref{ass2} ensures that
$Z_\nu(v)\in X$. Indeed, if $|v_i|<\nu$, then
$\nu<\omega$ implies $0\in[l_i,u_i]$, where $\omega$ is defined as in \eqref{new}. Consequently,
$Z_\nu(v)\in\mathcal G_\nu$ and $d^{Z_\nu(v)}=d^v$.

\begin{algorithm}[h]
\caption{Restarted Accelerated Proximal ADMM (RA-pADMM)}
\label{alg:main}
\small
\begin{algorithmic}
\State \textbf{Initialization:}
$(x^0,p^0)\in X\times\R^m$,
$\alpha\geq2$, $\rho\in(0,2)$, $\sigma>0$; choose
$\eta>\norm{A}^2/\min\{1,2-\rho\}^2$, set
$\tau=(\sigma\eta)^{-1}$, and choose
$0<\delta<\lambda-\nu L_f\norm{A}$.

\Statex \textbf{Phase I: Branch identification}
\State Set $w^0=(x^0,p^0)$,
$\bar w^{-1}=(\bar x^{-1},\bar p^{-1}):=w^0$, and $j=0$.
\While{true}
    \State Set $d^j=d^{\bar x^{j-1}}$ according to
    \eqref{lmdk-1} and compute
    $\bar w^j=S_{d^j}(w^j)$ from
    \eqref{eq:S-map}.
    \If{$\bar x^j\in\mathcal G_\nu$}
        \State Set
        $w^{0,0}=(x^{0,0},p^{0,0}):=\bar w^j$
        and \textbf{break}.
    \EndIf
    \State Update
    \Statex \hspace*{\dimexpr\linewidth-\textwidth\relax}%
    \begin{minipage}{\textwidth}
    \begin{equation}\label{eq:halpern-update}
      w^{j+1}
      =\frac{1}{j+2}w^0
      +\frac{j+1}{j+2}\bar w^j.
    \end{equation}
    \end{minipage}
    \State Set $j\leftarrow j+1$.
\EndWhile

\Statex \textbf{Phase II: Restarted acceleration}
\State Set $s=k=0$ and initialize
$d^{0,0}=d^{x^{0,0}}$,
$V^0=F(x^{0,0})$, and
$\widehat w^{0,0}=w^{0,0}$.
\While{the stopping condition is not satisfied}
    \State Compute
    $\bar w^{s,k}=S_{d^{s,k}}(w^{s,k})$
    from \eqref{eq:S-map}.
    \State Compute
    $\widetilde x^{s,k}=Z_\nu(\bar x^{s,k})$ and
    $\widetilde d^{s,k}=d^{\widetilde x^{s,k}}$
    according to \eqref{lmdk-1}.
    \If{$\widetilde d^{s,k}\neq d^{s,k}$ and
        $F(\widetilde x^{s,k})\leq V^s-\delta$}
        \State Set
        \(
          w^{s+1,0}
          =(x^{s+1,0},p^{s+1,0})
          :=(\widetilde x^{s,k},\bar p^{s,k}).
        \)
        \State Set $s\leftarrow s+1$,
        $d^{s,0}=d^{x^{s,0}}$,
        $V^s=F(x^{s,0})$,
        $\widehat w^{s,0}=w^{s,0}$, and $k=0$.
    \Else
        \State Update
        \Statex \hspace*{\dimexpr\linewidth-\textwidth\relax}%
        \begin{minipage}{\textwidth}
        \[
          \widehat w^{s,k+1}
          =w^{s,k}
          +\rho(\bar w^{s,k}-w^{s,k}).
        \]
        \end{minipage}
        \State Update
        \Statex \hspace*{\dimexpr\linewidth-\textwidth\relax}%
        \begin{minipage}{\textwidth}
        \begin{equation}\label{fkm}
        \begin{aligned}
          w^{s,k+1}
          =w^{s,k}
          +\frac{\alpha}{2(k+\alpha)}
          \bigl(\widehat w^{s,k+1}-w^{s,k}\bigr)+\frac{k}{k+\alpha}
          \bigl(\widehat w^{s,k+1}
          -\widehat w^{s,k}\bigr).
        \end{aligned}
        \end{equation}
        \end{minipage}
        \State Set $d^{s,k+1}=d^{s,k}$ and
        $k\leftarrow k+1$.
    \EndIf
\EndWhile
\end{algorithmic}
\normalsize
\end{algorithm}

\subsection{Convergence analysis}

Let $\{w^j\}$ and $\{\bar w^j\}$ denote the Phase I iterates generated by
the RA-pADMM algorithm. For each Phase II epoch $s$, let
$\{w^{s,k}\}$, $\{\bar w^{s,k}\}$, and
$\{\widehat w^{s,k}\}$ denote the corresponding epochwise iterates.
In this subsection, we establish the convergence of the algorithm
and derive convergence rates for the objective value gap and the
KKT residual.

\begin{lemma}
\label{finite Phase I}
There exists a $J\in\mathbb N$ such that
$\bar x^J\in\mathcal G_\nu$.
Consequently, Phase~I terminates after finitely many
evaluations of $S_{d^j}$.
\end{lemma}

\begin{proof}
Suppose that Phase~I continues indefinitely, that is,
$\bar x^j\notin\mathcal G_\nu$ for all $j\geq0$.
The Halpern update \eqref{eq:halpern-update} gives
\begin{equation}\label{x_update}
  x^j
  =\frac{1}{j+1}x^0
   +\frac{j}{j+1}\bar x^{j-1},
  \qquad j\geq1.
\end{equation}
Since $x^0\in X$ and every proximal output $\bar x^j$ belongs to $X$,
we have $x^j\in X$ for all $j\geq0$. For each $j$, define $\xi^j$ and $q^j$ as in
Remark~\ref{rem:x-explicit} with $w=w^j$.
By the optimality condition for the
$\bar y^j$-update in \eqref{eq:padmm-explicit},
we have $\xi^j\in\partial f(\bar y^j)$.
By Assumption~\ref{ass1}, $\|\xi^j\|\leq L_f$.
Hence,
\begin{equation}\label{eq:q-x-bound}
  |q_i^j-x_i^j|
  =\tau|(A^\top\xi^j)_i|
  \leq\tau L_f\|A\|,
  \qquad i=1,\ldots,n.
\end{equation}
Set $\Delta:=\tau(\lambda/\nu-L_f\|A\|)>0$ and
$M_0:=\max_{1\leq i\leq n}|x_i^0|$, where the positivity of
$\Delta$ follows from Assumption~\ref{ass2}.
Choose an integer $J_0\geq1$ such that
\[
  \frac{M_0+\nu}{j+1}
  \leq\min\left\{\frac{\Delta}{2},\frac{\nu}{2}\right\},
  \qquad j\geq J_0.
\]

Fix $j\geq J_0$.
If $d_i^j=1$, then $|\bar x_i^{j-1}|<\nu$.
Since $\bar x_i^{j-1}\in[l_i,u_i]$ and $\nu<\omega$,
we have $0\in[l_i,u_i]$.
Moreover, \eqref{x_update} gives
\begin{equation}\label{eq:xk-bound}
  |x_i^j|
  \leq
  \frac{|x_i^0|}{j+1}
  +\frac{j}{j+1}|\bar x_i^{j-1}|
  \leq
  |\bar x_i^{j-1}|+\frac{\Delta}{2}.
\end{equation}
Then, using \eqref{eq:sparse-update}, we obtain
\begin{equation}\label{eq:phase-one-contraction}
\begin{aligned}
  |\bar x_i^j|
  \leq
  \left(|q_i^j|-\frac{\tau\lambda}{\nu}\right)_+
  \leq (|x_i^j|-\Delta)_+
  \leq
  \left(|\bar x_i^{j-1}|-\frac{\Delta}{2}\right)_+
  <\nu,
\end{aligned}
\end{equation}
where the second and third inequalities follow from
\eqref{eq:q-x-bound} and \eqref{eq:xk-bound}, respectively.
Thus, $d_i^{j+1}=1$.

If $d_i^j=2$, then $\bar x_i^{j-1}\geq\nu$, and
\[
  x_i^j
  \geq
  -\frac{M_0}{j+1}+\frac{j\nu}{j+1}
  =\nu-\frac{M_0+\nu}{j+1}
  \geq\frac{\nu}{2}.
\]
Consequently, we have
\[
  q_i^j
  \geq x_i^j-\tau L_f\|A\|
  \geq\frac{\nu}{2}-\tau L_f\|A\|
  >-\frac{2\tau\lambda}{\nu}.
\]
Since $u_i\geq\bar x_i^{j-1}\geq\nu$,
\eqref{eq:sparse-update} yields $\bar x_i^j\geq0$,
and hence $d_i^{j+1}\in\{1,2\}$.
Similarly, $d_i^j=3$ implies
$d_i^{j+1}\in\{1,3\}$. Since the number of coordinates is
finite, there exist $J_d\geq J_0$ and
$\widehat d\in\mathcal D$ such that
$d^j=\widehat d$ for all $j\geq J_d$.

For every coordinate with $\widehat d_i=1$, repeated application of
\eqref{eq:phase-one-contraction} yields
\[
  |\bar x_i^{J_d+r}|
  \leq
  \left(
    |\bar x_i^{J_d-1}|
    -\frac{(r+1)\Delta}{2}
  \right)_+,
  \qquad r\geq0.
\]
Since $|\bar x_i^{J_d-1}|<\nu$, it follows that
$\bar x_i^{J_d+r}=0$ whenever
$r+1\geq\lceil 2\nu/\Delta\rceil$.
Hence, these coordinates become zero after at most
$\lceil 2\nu/\Delta\rceil$ iterations starting from $J_d$
and remain zero thereafter.
For the remaining coordinates with $\widehat d_i\in\{2,3\}$,
branch stabilization yields $|\bar x_i^j|\geq\nu$
for all $j\geq J_d$.
Thus, with $J:=J_d+\lceil 2\nu/\Delta\rceil$, we have
$\bar x^J\in\mathcal G_\nu$, contradicting the
assumption that Phase~I continues indefinitely.
Hence, Phase~I terminates after finitely many evaluations
of $S_{d^j}$.
\qed
\end{proof}

Once Phase~I terminates, all subsequent branch changes can occur only
through accepted restarts in Phase~II. The next lemma shows that the
number of such restarts is finite.

\begin{lemma}
\label{lem:finite-restarts}
Let $R_{\mathrm{res}}$ denote the number of accepted restarts in
Phase~II. Then
\begin{equation}\label{eq:restart-bound}
  R_{\mathrm{res}}
  \leq
  \left\lceil
  \frac{V^0-\mathcal F_{\inf}}{\delta}
  \right\rceil.
\end{equation}
Hence, only finitely many restarts can occur in Phase~II, and there exists a final epoch indexed by $\bar s$ with fixed branch $\bar d=d^{\bar s,0}$.
\end{lemma}

\begin{proof}
By Remark~\ref{rem:branch-solvability},
$\mathcal F_{\inf}>-\infty$.
For any $r\in\mathbb N$ such that $r$ restarts have been
accepted, the definition of $V^r$ and the restart criterion give
\(
  \mathcal F_{\inf}\leq V^r\leq V^0-r\delta.
\)
Thus, every such $r$ satisfies
\[
  r\leq \left\lceil\frac{V^0-\mathcal F_{\inf}}{\delta} \right\rceil,
\]
so only finitely many restarts can be accepted, and
\eqref{eq:restart-bound} follows. Since $R_{\mathrm{res}}$ is finite, there exists a final epoch
indexed by $\bar s$. The branch remains unchanged within each
epoch, so the final epoch has fixed branch
$\bar d=d^{\bar s,0}$.
\qed
\end{proof}

By Lemmas~\ref{finite Phase I} and~\ref{lem:finite-restarts},
the RA-pADMM algorithm reaches the final epoch indexed by $\bar s$
after finitely many Phase~I iterations and accepted restarts in
Phase~II.
Since the branch remains unchanged within this epoch,
$d^{\bar s,k}=\bar d$ for every $k\geq0$.
Consequently, all iterations in the final epoch are generated by
the fixed branch map $S_{\bar d}$. Let
\[
  T_{\bar d,\rho}:=I+\rho(S_{\bar d}-I).
\]
As observed in
\cite[Remark~2.6]{SunYuanZhangZhao2025}, when
$\alpha=2$, the iteration in the final epoch can be
written as the Halpern iteration
\begin{equation}\label{eq:final-halpern}
  w^{\bar s,k+1}
  =
  \frac{1}{k+2}w^{\bar s,0}
  +
  \frac{k+1}{k+2}
  T_{\bar d,\rho}(w^{\bar s,k}).
\end{equation}
When $\alpha>2$, the final epoch iterates follow the unit-stepsize
fast KM scheme applied to $T_{\bar d,\rho}$.

Thus, after the last restart, the algorithm reduces to a Halpern or
fast KM iteration applied to the fixed nonexpansive mapping
$T_{\bar d,\rho}$, whose fixed-point set satisfies
$\Fix T_{\bar d,\rho}=\Fix S_{\bar d}\neq\varnothing$.
The resulting convergence properties are summarized in the following
proposition.

\begin{proposition}
\label{prop:final-epoch-rate}
Let $\bar s$ and $\bar d$ be as in
Lemma~\ref{lem:finite-restarts}. Then there exists $w^*\in\Fix T_{\bar d,\rho}
=\Fix S_{\bar d}$ such that
\begin{equation}\label{eq:final-output-convergence}
  w^{\bar s,k}\to w^*,
  \qquad
  \bar w^{\bar s,k}\to w^*.
\end{equation}
Moreover, the fixed-point residual satisfies the following
convergence rates:
\begin{equation}\label{eq:final-residual-rates}
  \|\bar w^{\bar s,k}-w^{\bar s,k}\|_H
  =
  \begin{cases}
    O(k^{-1}), & \alpha=2,\\[1mm]
    o(k^{-1}), & \alpha>2.
  \end{cases}
\end{equation}
\end{proposition}

\begin{proof}
If $\alpha=2$, \eqref{eq:final-halpern} is the
Halpern iteration applied to $T_{\bar d,\rho}$.
By Lemma~\ref{T-nonexpansive}, $T_{\bar d,\rho}$ is nonexpansive.
Moreover, Assumption~\ref{ass3} ensures that
$\Fix T_{\bar d,\rho}\neq\varnothing$.
Then, by \cite[Theorem~2]{Wittmann1992},
$w^{\bar s,k}\to w^*\in\Fix T_{\bar d,\rho}$, while
\cite[Theorem~2.1]{Lieder2021} gives
$\|w^{\bar s,k}-T_{\bar d,\rho}w^{\bar s,k}\|_H=O(k^{-1})$.
Since $I-T_{\bar d,\rho}=\rho(I-S_{\bar d})$, we have
\(
  \|\bar w^{\bar s,k}-w^{\bar s,k}\|_H=O(k^{-1})
\),
which together with $w^{\bar s,k}\to w^*$ implies that
$\bar w^{\bar s,k}\to w^*$.

If $\alpha>2$, the final epoch is the unit-stepsize fast KM
iteration applied to $T_{\bar d,\rho}$. By
\cite[Theorems~3.4 and~3.5]{BotNguyen2023}, we obtain
$w^{\bar s,k}\to w^*\in\Fix T_{\bar d,\rho}$ and
$\|w^{\bar s,k}-T_{\bar d,\rho}w^{\bar s,k}\|_H=o(k^{-1})$.
As in the case $\alpha=2$, since
$I-T_{\bar d,\rho}=\rho(I-S_{\bar d})$, it follows that
$\|\bar w^{\bar s,k}-w^{\bar s,k}\|_H=o(k^{-1})$, and hence
$\bar w^{\bar s,k}\to w^*$.
\qed
\end{proof}

Proposition~\ref{prop:final-epoch-rate} establishes convergence on the
fixed final branch. Its connection with the original capped-$\ell_1$
problem is established by the eventual branch consistency condition
$d^{\bar x^{\bar s,k}}=\bar d$.

\begin{lemma}
\label{lem:eventual-branch-consistency}
Let $\bar s$ and $\bar d$ be as defined in Lemma~\ref{lem:finite-restarts}. There exists a finite index $K_1\geq1$ such that
\begin{equation}\label{eq:eventual-branch-consistency}
  d^{\bar x^{\bar s,k}}=\bar d,
  \qquad k\geq K_1.
\end{equation}
\end{lemma}

\begin{proof}
By the definition of the final epoch and
$\bar d=d^{\bar s,0}=d^{x^{\bar s,0}}$, we have
$V^{\bar s}=F(x^{\bar s,0})=F_{\bar d}(x^{\bar s,0})$.
By Proposition~\ref{prop:final-epoch-rate},
$\bar x^{\bar s,k}\to x^*$, where $x^*$ minimizes
$F_{\bar d}$ over $X$. Since every proximal output
$\bar x^{\bar s,k}$ belongs to $X$ and $F_{\bar d}$ is
continuous on $X$, we have
\[
  F_{\bar d}(\bar x^{\bar s,k})
  \to F_{\bar d}(x^*)
  =\min_{x\in X}F_{\bar d}(x)
  \leq V^{\bar s}.
\]
Thus, there exists a sufficiently large $K_1\geq 1$ such that,
for all $k\geq K_1$, we have
\begin{equation}\label{eq:final-value-bound}
  F_{\bar d}(\bar x^{\bar s,k})
  \leq
  V^{\bar s}+\lambda-\nu L_f\|A\|-\delta.
\end{equation}

We first show that $\bar d_i=1$ implies
$\bar x_i^{\bar s,k}=0$ for all $k\geq0$. Since the initial point of the final epoch belongs to
$\mathcal G_\nu$, $\bar d_i=1$ implies $x_i^{\bar s,0}=0$. Define $\xi^{\bar s,0}$ and $q^{\bar s,0}$ as in
Remark~\ref{rem:x-explicit} with $w=w^{\bar s,0}$.
By an argument similar to that used to derive
\eqref{eq:q-x-bound}, we have, for every $i$ with
$\bar d_i=1$,
$|q_i^{\bar s,0}|
=\tau|[A^\top\xi^{\bar s,0}]_i|
\leq\tau L_f\|A\|
<\tau\lambda/\nu$.
It follows from \eqref{eq:sparse-update} that
$\bar x_i^{\bar s,0}=0$, and hence
$\widehat x_i^{\bar s,1}=0$.
Since $\widehat x_i^{\bar s,0}
=x_i^{\bar s,0}=0$,
\eqref{fkm} gives $x_i^{\bar s,1}=0$.
Repeating this argument inductively yields
$x_i^{\bar s,k}
=\bar x_i^{\bar s,k}
=\widehat x_i^{\bar s,k}=0$
for all $k\geq0$.

Suppose, to the contrary, that there exists some $K_2\geq K_1$
such that $d^{\bar x^{\bar s,K_2}}\neq\bar d$.
Let
$\widetilde x^{\bar s,K_2}=Z_\nu(\bar x^{\bar s,K_2})$.
Since $\bar d_i=1$ implies
$\bar x_i^{\bar s,K_2}=0$,
any branch inconsistency must occur at an index
$i$ with $\bar d_i\in\{2,3\}$.
We define the following two index sets
to characterize these inconsistencies:
\[
\begin{aligned}
  I_0^{\bar s,K_2}
  &:=
  \left\{
  i:\bar d_i\in\{2,3\},\
  |\bar x_i^{\bar s,K_2}|<\nu
  \right\},\\
  I_{\mathrm{flip}}^{\bar s,K_2}
  &:=
  \left\{
  i:
  \bar d_i=2,\ \bar x_i^{\bar s,K_2}\leq-\nu
  \ \text{or}\
  \bar d_i=3,\ \bar x_i^{\bar s,K_2}\geq\nu
  \right\}.
\end{aligned}
\]
Then $d^{\bar x^{\bar s,K_2}}\neq\bar d$ implies
$I_0^{\bar s,K_2}\cup
I_{\mathrm{flip}}^{\bar s,K_2}\neq\varnothing$.
For each $i$, set
\[
  \gamma_i:=
  \lambda\left(
  \frac{|\bar x_i^{\bar s,K_2}|}{\nu}
  -\theta_{\bar d_i}(\bar x_i^{\bar s,K_2})
  -\min\left\{
  \frac{|\widetilde x_i^{\bar s,K_2}|}{\nu},1
  \right\}
  \right).
\]
We have $\gamma_i\geq\lambda$ for
$i\in I_0^{\bar s,K_2}$,
$\gamma_i\geq2\lambda$ for
$i\in I_{\mathrm{flip}}^{\bar s,K_2}$, and
$\gamma_i\geq0$ for all remaining indices. Hence,
\begin{equation}\label{eq:final-value-difference_1}
\begin{aligned}
  &\lambda\Phi^{\bar d}(\bar x^{\bar s,K_2})
  -\lambda\Phi(\widetilde x^{\bar s,K_2})\\
  &\quad=
  \sum_{i=1}^n\gamma_i
  \geq
  \lambda|I_0^{\bar s,K_2}|
  +2\lambda|I_{\mathrm{flip}}^{\bar s,K_2}|
  \geq\lambda,
\end{aligned}
\end{equation}
where $|I|$ denotes the cardinality of a set $I$, and
the last inequality follows from
$I_0^{\bar s,K_2}\cup
I_{\mathrm{flip}}^{\bar s,K_2}\neq\varnothing$.
Moreover, since $Z_\nu$ changes $\bar x^{\bar s,K_2}$ only at indices in $I_0^{\bar s,K_2}$, we have
$\|\widetilde x^{\bar s,K_2}-\bar x^{\bar s,K_2}\|_1
=\sum_{i\in I_0^{\bar s,K_2}}|\bar x_i^{\bar s,K_2}|
\leq\nu|I_0^{\bar s,K_2}|$.
Therefore, we obtain
\begin{equation}\label{eq:final-value-difference_2}
\begin{aligned}
  f(A\widetilde x^{\bar s,K_2})
  -f(A\bar x^{\bar s,K_2})
  &\leq
  L_f\|A\|
  \|\widetilde x^{\bar s,K_2}
  -\bar x^{\bar s,K_2}\|\\
  &\leq
  L_f\|A\|
  \|\widetilde x^{\bar s,K_2}
  -\bar x^{\bar s,K_2}\|_1\\
  &\leq
  \nu L_f\|A\||I_0^{\bar s,K_2}|.
\end{aligned}
\end{equation}
Combining \eqref{eq:final-value-difference_1} and
\eqref{eq:final-value-difference_2}, we have
\[
\begin{aligned}
F_{\bar d}(\bar x^{\bar s,K_2})
-F(\widetilde x^{\bar s,K_2})
&\geq
(\lambda-\nu L_f\|A\|)
|I_0^{\bar s,K_2}|+
2\lambda|I_{\mathrm{flip}}^{\bar s,K_2}|\\
&\geq \lambda-\nu L_f\|A\|.
\end{aligned}
\]
Together with \eqref{eq:final-value-bound}, this yields
$F(\widetilde x^{\bar s,K_2})\leq V^{\bar s}-\delta$.
Since
$d^{\widetilde x^{\bar s,K_2}}
=d^{\bar x^{\bar s,K_2}}
\neq\bar d=d^{\bar s,K_2}$, the restart criterion is
satisfied, contradicting the definition of the final epoch.
Hence, no such $K_2\geq K_1$ satisfying
$d^{\bar x^{\bar s,K_2}}\neq\bar d$ exists; that is,
$d^{\bar x^{\bar s,k}}=\bar d$ for all $k\geq K_1$.
\qed
\end{proof}

For the final branch $\bar d$ and an input $w=(x,p)$,
let $(\bar y,\bar x,\bar p)$ be generated by
\eqref{eq:padmm-explicit}. Then
$(\bar x,\bar p)=S_{\bar d}(w)$ by
Proposition~\ref{prop:padmm-fixed-point}.
Define the associated KKT residual by
\begin{equation}\label{eq:kkt-residual}
  \mathcal E_{\bar d}(w):=
  \dist\bigl(0,\partial g_{\bar d}(\bar x)+A^\top\bar p\bigr)
  +\dist\bigl(\bar p,\partial f(\bar y)\bigr)
  +\norm{A\bar x-\bar y}.
\end{equation}
The following theorem establishes the convergence rates of the
KKT residual and the objective value.

\begin{theorem}
\label{thm:convergence-rates}
Let $\bar s$ and $\bar d$ be as in
Lemma~\ref{lem:finite-restarts}, and let $K_1$ be as in
Lemma~\ref{lem:eventual-branch-consistency}. Then the sequences
$\{x^{\bar s,k}\}$ and $\{\bar x^{\bar s,k}\}$ converge to a
common lifted stationary point $x^*$ of problem~\eqref{model}.
Moreover,
\begin{equation}\label{eq:kkt-objective-rates}
\begin{aligned}
  \mathcal E_{\bar d}(w^{\bar s,k})
  &=
  \begin{cases}
    O(k^{-1}), & \alpha=2,\\
    o(k^{-1}), & \alpha>2,
  \end{cases}
  && k\geq1,\\[0.6ex]
  F(\bar x^{\bar s,k})-F(x^*)
  &=
  \begin{cases}
    O(k^{-1}), & \alpha=2,\\
    o(k^{-1}), & \alpha>2,
  \end{cases}
  && k\geq K_1.
\end{aligned}
\end{equation}
\end{theorem}

\begin{proof}
By Proposition~\ref{prop:final-epoch-rate}, the sequences
$\{w^{\bar s,k}\}$ and $\{\bar w^{\bar s,k}\}$ converge to a
common point
$w^*=(x^*,p^*)\in\Fix S_{\bar d}$.
Hence, both $\{x^{\bar s,k}\}$ and
$\{\bar x^{\bar s,k}\}$ converge to $x^*$.
For $k\geq K_1$,
Lemma~\ref{lem:eventual-branch-consistency} gives
$d^{\bar x^{\bar s,k}}=\bar d$.
Together with $\bar x^{\bar s,k}\to x^*$, this yields
$\bar d\in\mathcal D^{x^*}$. Hence,
$\Phi^{\bar d}(\bar x^{\bar s,k})
=\Phi(\bar x^{\bar s,k})$ for $k\geq K_1$ and
$\Phi^{\bar d}(x^*)=\Phi(x^*)$.

Since $w^*=(x^*,p^*)\in\Fix S_{\bar d}$,
Proposition~\ref{prop:padmm-fixed-point} implies that
$(x^*,Ax^*,p^*)$ satisfies the KKT system
\eqref{eq:branch-kkt} for branch $\bar d$. Together with
$\bar d\in\mathcal D^{x^*}$,
Remark~\ref{eq_lifed} then implies that $x^*$ is a lifted
stationary point of problem~\eqref{model}.

For $k\geq 1$, the first-order optimality conditions for the
proximal subproblems in \eqref{eq:padmm-explicit}, together with
the definitions of $\xi^{\bar s,k}$ and $\bar p^{\bar s,k}$,
yield
$\tau^{-1}(x^{\bar s,k}-\bar x^{\bar s,k})
-A^\top\xi^{\bar s,k}
\in\partial g_{\bar d}(\bar x^{\bar s,k})$,
$\xi^{\bar s,k}\in\partial f(\bar y^{\bar s,k})$,
$\bar p^{\bar s,k}-\xi^{\bar s,k}
=\sigma A(\bar x^{\bar s,k}-x^{\bar s,k})$, and
$\bar p^{\bar s,k}-p^{\bar s,k}
=\sigma(A\bar x^{\bar s,k}-\bar y^{\bar s,k})$.
Hence, we have
\[
\begin{aligned}
  \mathcal E_{\bar d}(w^{\bar s,k})
  \leq{}&
  \left[\frac{1}{\tau}
  +\sigma(\norm{A}^2+\norm{A})\right]
  \norm{\bar x^{\bar s,k}-x^{\bar s,k}}\\
  &+\frac{1}{\sigma}
  \norm{\bar p^{\bar s,k}-p^{\bar s,k}}.
\end{aligned}
\]
Moreover, by the definition of $H$, it holds that
\[
\begin{aligned}
  \norm{\bar w^{\bar s,k}-w^{\bar s,k}}_H^2
  =
  \frac{1}{\tau}
  \norm{\bar x^{\bar s,k}-x^{\bar s,k}}^2+\frac{1}{\sigma}
  \norm{\bar p^{\bar s,k}-p^{\bar s,k}}^2,
\end{aligned}
\]
which implies
\[
\begin{aligned}
  \norm{\bar x^{\bar s,k}-x^{\bar s,k}}
  \leq
  \sqrt{\tau}
  \norm{\bar w^{\bar s,k}-w^{\bar s,k}}_H,\quad
  \norm{\bar p^{\bar s,k}-p^{\bar s,k}}
  \leq
  \sqrt{\sigma}
  \norm{\bar w^{\bar s,k}-w^{\bar s,k}}_H.
\end{aligned}
\]
Therefore,
\begin{equation}\label{eq:kkt-residual-bound}
  \mathcal E_{\bar d}(w^{\bar s,k})
  \leq
  \left[
    \frac{1}{\sqrt{\tau}}
    +\sigma\sqrt{\tau}
    \bigl(\norm{A}^2+\norm{A}\bigr)
    +\frac{1}{\sqrt{\sigma}}
  \right]
  \norm{\bar w^{\bar s,k}-w^{\bar s,k}}_H.
\end{equation}
Combining Proposition~\ref{prop:final-epoch-rate} with
\eqref{eq:kkt-residual-bound} yields the first rate in
\eqref{eq:kkt-objective-rates}.

Set
$r_1^k:=\dist(0,\partial g_{\bar d}(\bar x^{\bar s,k})
+A^\top\bar p^{\bar s,k})$,
$r_2^k:=\dist(\bar p^{\bar s,k},
\partial f(\bar y^{\bar s,k}))$, and
$r_3^k:=\norm{A\bar x^{\bar s,k}-\bar y^{\bar s,k}}$.
Then
$\mathcal E_{\bar d}(w^{\bar s,k})=r_1^k+r_2^k+r_3^k$.
Choose $\mu^k\in\partial g_{\bar d}(\bar x^{\bar s,k})$ and
$\zeta^k\in\partial f(\bar y^{\bar s,k})$ such that
$\|\mu^k+A^\top\bar p^{\bar s,k}\|=r_1^k$ and
$\|\bar p^{\bar s,k}-\zeta^k\|=r_2^k$. Since both $\bar x^{\bar s,k}$ and $x^*$ belong to $X$, we have
\[
\begin{aligned}
0\leq{}&
F_{\bar d}(\bar x^{\bar s,k})-F_{\bar d}(x^*)\\
={}&
g_{\bar d}(\bar x^{\bar s,k})-g_{\bar d}(x^*)
+f(A\bar x^{\bar s,k})-f(Ax^*)\\
\leq{}&
\langle\mu^k,\bar x^{\bar s,k}-x^*\rangle
+f(\bar y^{\bar s,k})-f(Ax^*)
+L_f\|A\bar x^{\bar s,k}-\bar y^{\bar s,k}\|\\
\leq{}&
\langle\mu^k,\bar x^{\bar s,k}-x^*\rangle
+\langle\zeta^k,\bar y^{\bar s,k}-Ax^*\rangle
+L_f r_3^k\\
={}&
\langle\mu^k+A^\top\zeta^k,
\bar x^{\bar s,k}-x^*\rangle
+\langle\zeta^k,
\bar y^{\bar s,k}-A\bar x^{\bar s,k}\rangle
+L_f r_3^k\\
\leq{}&
\|\mu^k+A^\top\zeta^k\|
\|\bar x^{\bar s,k}-x^*\|
+\|\zeta^k\|r_3^k+L_f r_3^k\\
\leq{}&
\|\bar x^{\bar s,k}-x^*\|
\left(
\|\mu^k+A^\top\bar p^{\bar s,k}\|
+\|A\|\|\zeta^k-\bar p^{\bar s,k}\|
\right)
+(\|\zeta^k\|+L_f)r_3^k\\
\leq{}&
\|\bar x^{\bar s,k}-x^*\|
\bigl(r_1^k+\|A\|r_2^k\bigr)
+2L_f r_3^k.
\end{aligned}
\]
The nonnegativity follows from the optimality of $x^*$ for
$F_{\bar d}$ over $X$. The first upper bound follows from
$\mu^k\in\partial g_{\bar d}(\bar x^{\bar s,k})$ and the
$L_f$-Lipschitz continuity of $f$, while the next follows from
$\zeta^k\in\partial f(\bar y^{\bar s,k})$.
Since $\bar x^{\bar s,k}\to x^*$, there exists $c_0>0$ such that
$F_{\bar d}(\bar x^{\bar s,k})-F_{\bar d}(x^*)
\leq c_0\mathcal E_{\bar d}(w^{\bar s,k})$ for all
$k\geq 1$. Using
$\Phi^{\bar d}(\bar x^{\bar s,k})
=\Phi(\bar x^{\bar s,k})$ for $k\geq K_1$ and
$\Phi^{\bar d}(x^*)=\Phi(x^*)$, we obtain
\[
  0\leq
  F(\bar x^{\bar s,k})-F(x^*)
  \leq
  c_0\mathcal E_{\bar d}(w^{\bar s,k}),
  \qquad k\geq K_1.
\]
The second rate in \eqref{eq:kkt-objective-rates}
then follows from the first.
\qed
\end{proof}

\begin{corollary}\label{cor:strong-local-minimizer}
Both sequences $\{x^{\bar s,k}\}$ and
$\{\bar x^{\bar s,k}\}$ converge to a common
$\nu$-strong local minimizer $x^*$ of
problem~\eqref{model_l0}.
\end{corollary}

\begin{proof}
By Theorem~\ref{thm:convergence-rates}, the sequences
$\{x^{\bar s,k}\}$ and $\{\bar x^{\bar s,k}\}$ converge to a common lifted stationary point $x^*$ of problem~\eqref{model}.
Then by \cite[Proposition~2.7]{bian2020}, we have that $x^*$ is a
$\nu$-strong local minimizer of problem~\eqref{model_l0}.
\qed
\end{proof}

\section{Numerical experiments}
\label{sec:numerical}

In this section, we evaluate the convergence behavior and solution
quality of the RA-pADMM algorithm. We compare the proposed method with the accelerated nested proximal-gradient method (ANPG) of Tao et al.~\cite{TaoBianQu2026}. We also include pADMM as a baseline, with updates
$\bar w^k=S_{d^k}(w^k)$ and $w^{k+1}=\bar w^k$,
where $d^k=d^{x^k}$. We implement the RA-pADMM algorithm with different choices of
$\alpha$ and $\rho$ in Phase~II. For convenience, the variants
corresponding to $(\alpha,\rho)=(2,1)$, $\alpha=2$ and
$\rho\neq1$, $\alpha>2$ and $\rho=1$, and $\alpha>2$ and
$\rho\neq1$ are referred to as Halpern pADMM (H-pADMM),
relaxed Halpern pADMM (RH-pADMM), fast KM pADMM
(FKM-pADMM), and relaxed fast KM pADMM (RFKM-pADMM),
respectively. The
values of $\rho$ and $\alpha$ in each experiment are given. For the restarted accelerated pADMM methods, the sufficient decrease
parameter is set to
\(
    \delta=\frac{\lambda-\nu L_f\|A\|}{2}.
\)

We first use randomly generated sparse recovery problems with a
general matrix $A$ to compare convergence behavior and recovery
accuracy. We then consider sparse support vector machines on
real-world data sets to evaluate classification accuracy, sparsity,
and computation time. All algorithms are implemented in Python and the experiments are conducted on a Windows PC equipped with an Intel Core i5-13500H processor and 16 GB of RAM.

\subsection{Synthetic sparse recovery}
\label{subsec:synthetic}

We first consider the following sparse recovery problem with a general
linear operator and an $\ell_1$ loss:
\begin{equation}
\label{eq:exp-lad}
    \min_{x\in[0,10]^n}
    \frac{1}{m}\|Ax-c\|_1
    +\lambda\sum_{i=1}^n
    \min\left\{\frac{|x_i|}{\nu},1\right\}.
\end{equation}

We set $m=140$ and $n=160$. The ground-truth vector
$x_{\mathrm{true}}$ has $10$ nonzero entries. Its support is selected
uniformly at random without replacement, and the corresponding nonzero
values are independently sampled from the uniform distribution on
$[2,10]$. The entries of $A$ are independently generated from the
standard normal distribution. The resulting matrix is normalized by
its spectral norm so that $\|A\|=1$. The observation vector is generated
according to
\(
    c=Ax_{\mathrm{true}}+\varepsilon,
\)
where the components of $\varepsilon$ are independent Gaussian random
variables with mean zero and standard deviation $1\times 10^{-2}$.

All methods solve the same problem with $\lambda=10$,
$\nu=1$, and $X=[0,10]^n$. For $f(y)=\|y-c\|_1/m$, we have $L_f=1/\sqrt{m}$. For the pADMM-based methods,
we set $\sigma=1.2\times10^{-3}$ and $p^0=0$.
The parameters $\eta$ and $\tau$ are chosen as
\[
    \eta=
    \frac{(1+\epsilon_0)\|A\|^2}
         {\min\{1,2-\rho\}^2},
    \qquad
    \tau=\frac{1}{\sigma\eta},
\]
where $\epsilon_0>0$ is a small positive constant. In this experiment,
we take $\epsilon_0=10^{-4}$. For the relaxed variants RH-pADMM and
RFKM-pADMM, we take $\rho=0.9$, while the fast KM variants use
$\alpha=3$. For ANPG, the initial smoothing parameter
is set to $\mu_0=50$. All methods are initialized at
$x^0=1.97\mathbf{1}$, where $\mathbf{1}$ denotes the all-ones column
vector.

All methods are run for 200 counted evaluations or updates, and the experiment is
repeated on $10$ independently generated problem instances. Since each
evaluation of $S_d$ and each proximal-gradient update require two
proximal-operator evaluations, which account for the main
computational cost, their numbers provide a fair basis for comparison.
For the pADMM-based methods, $k$ denotes the cumulative number of
evaluations of $S_d$, whereas for ANPG, it denotes the number of
proximal-gradient updates. To report at index $k$ the point obtained
after exactly $k$ such evaluations or updates, we set $z_r^0=x^0$
for all methods and, for $k\geq1$, define
\begin{equation}\label{eq:z_r^k}
z_r^k=
\begin{cases}
\bar{x}_r^{k-1}, & \text{for the pADMM-based methods},\\
x_{\mathrm{ANPG},r}^k, & \text{for ANPG}.
\end{cases}
\end{equation}

For support recovery, let $x_{\mathrm{true},r}$ denote the
ground-truth vector in the $r$th trial and define
\[
\mathcal{S}_r
:=\operatorname{supp}(x_{\mathrm{true},r}),
\qquad
\widehat{\mathcal{S}}_r^k
:=\{i: |(z_r^k)_i|\geq 10^{-3}\}.
\]
The corresponding $\mathrm{F1}$ score is
$\mathrm{F1}_r^k
:=
\frac{2|\widehat{\mathcal{S}}_r^k\cap\mathcal{S}_r|}
{|\widehat{\mathcal{S}}_r^k|+|\mathcal{S}_r|}.$
Then, $\mathrm{F1}_r^k=1$ if and only if
$\widehat{\mathcal{S}}_r^k=\mathcal{S}_r$. The average $\mathrm{F1}$ score over
the $10$ trials is defined by
\[
\mathrm{F1}_k
:=
\frac{1}{10}\sum_{r=1}^{10}\mathrm{F1}_r^k.
\]
For trial $r$, let $F_r$ denote the objective function associated
with the generated matrix and observation vector. We further define the average objective value deviation and average
relative distance by
\[
\mathrm{ObjDev}_k
:=
\frac{1}{10}\sum_{r=1}^{10}
\left|F_r(z_r^k)-F_r(x_{\mathrm{true},r})\right|
\]
and
\[
\mathrm{RelDist}_k
:=
\frac{1}{10}\sum_{r=1}^{10}
\frac{\|z_r^k-x_{\mathrm{true},r}\|_2}
{\|x_{\mathrm{true},r}\|_2},
\]
respectively.

\begin{figure}[ht]
    \centering
    \includegraphics[width=0.95\textwidth]
    {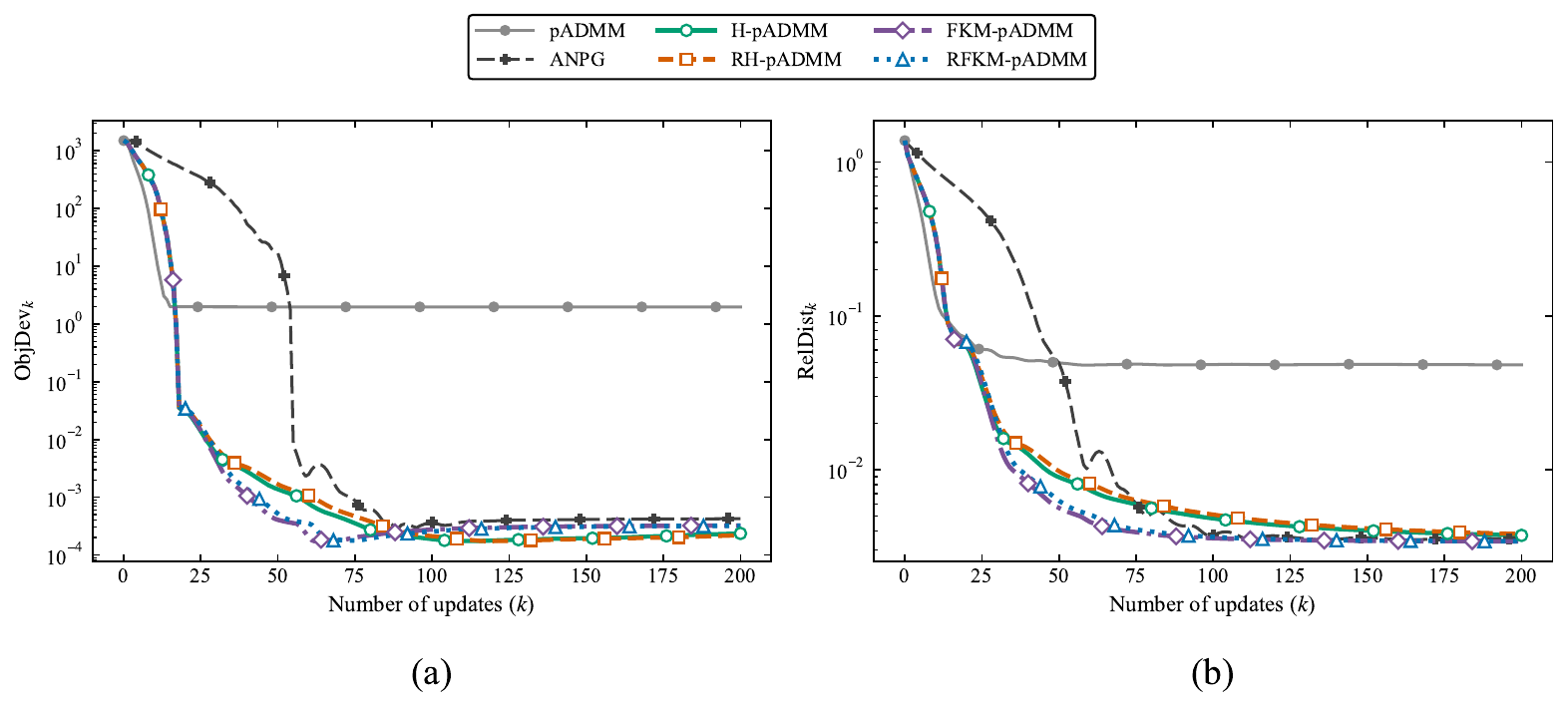}
    \caption{Average objective value deviation $\mathrm{ObjDev}_k$ and
    average relative distance $\mathrm{RelDist}_k$ for all methods}
    \label{fig:general_matrix_convergence}
\end{figure}

\begin{table}[ht]
    \centering
    \caption{Average performance of all methods after $200$ counted evaluations or updates.}
    \label{tab:general_matrix_results}
    \begin{tabular}{lccc}
        \toprule
        Method
        & $\mathrm{F1}_{200}$
        & $\mathrm{ObjDev}_{200}$
        & $\mathrm{RelDist}_{200}$ \\
        \midrule
        FKM-pADMM
        & \textbf{1.000}
        & $3.22\times10^{-4}$
        & $\mathbf{3.42\times10^{-3}}$ \\
        RFKM-pADMM
        & \textbf{1.000}
        & $3.20\times10^{-4}$
        & $3.43\times10^{-3}$ \\
        H-pADMM
        & \textbf{1.000}
        & $2.35\times10^{-4}$
        & $3.74\times10^{-3}$ \\
        RH-pADMM
        & \textbf{1.000}
        & $\mathbf{2.19\times10^{-4}}$
        & $3.82\times10^{-3}$ \\
        pADMM
        & 0.989
        & $1.97\times10^{0}$
        & $4.80\times10^{-2}$ \\
        ANPG
        & \textbf{1.000}
        & $4.24\times10^{-4}$
        & $3.56\times10^{-3}$ \\
        \bottomrule
    \end{tabular}
\end{table}

The values of $\mathrm{ObjDev}_k$ and $\mathrm{RelDist}_k$ are
plotted against the update count $k$ in
Fig.~\ref{fig:general_matrix_convergence}. The results in Fig.~\ref{fig:general_matrix_convergence}(a) indicate that
the objective value deviations of all four restarted accelerated pADMM
methods decrease more rapidly than the objective value deviation of ANPG.
FKM-pADMM and RFKM-pADMM exhibit the fastest decrease over most of the
iteration range. The four methods
also show a faster initial reduction in $\mathrm{RelDist}_k$ in
Fig.~\ref{fig:general_matrix_convergence}(b). Although ANPG
later decreases faster than H-pADMM and RH-pADMM, FKM-pADMM and
RFKM-pADMM maintain the fastest overall decrease.

Table~\ref{tab:general_matrix_results} reports
$\mathrm{F1}_{200}$, $\mathrm{ObjDev}_{200}$, and
$\mathrm{RelDist}_{200}$. Table~\ref{tab:general_matrix_results} further shows that RH-pADMM
achieves the smallest $\mathrm{ObjDev}_{200}$, whereas FKM-pADMM yields
the smallest $\mathrm{RelDist}_{200}$. The average $\mathrm{F1}$ score of pADMM is
strictly below $1$, indicating that it does not recover the exact
support in every trial. This
behavior is also reflected in
Fig.~\ref{fig:general_matrix_convergence}, where the pADMM curves in
both panels do not decrease to the levels reached by the other methods. By contrast, the four restarted accelerated
pADMM methods and ANPG attain an average $\mathrm{F1}$ score of $1$.

These results demonstrate the advantages of the restarted accelerated pADMM methods in accelerating convergence while maintaining accurate support recovery.

\subsection{Sparse classification on real data}
\label{subsec:real}

For binary classification, let $a_i\in\mathbb{R}^n$ be a feature
vector and let $y_i\in\{-1,1\}$ be its label. Define the $i$th row
of $A$ by $A_i=y_i a_i^\top$. We consider the capped-$\ell_1$
regularized support vector machine problem
\begin{equation}
\label{eq:hinge-capped}
    \min_{x\in[-20,20]^n}
    \frac{1}{m}\sum_{i=1}^m
    \bigl(1-(Ax)_i\bigr)_+
    +\lambda\sum_{j=1}^n
    \min\left\{\frac{|x_j|}{\nu},1\right\}.
\end{equation}

We use five binary-classification data sets from the UCI Machine
Learning Repository, including
Wisconsin Diagnostic Breast Cancer (WDBC), Ionosphere, Spambase,
quantitative structure--activity relationship (QSAR)
biodegradation, and Phishing Websites. Their sample sizes, original
feature dimensions, and parameter settings are reported in
Table~\ref{tab:svm_data_parameters}. For each data set, we generate ten independent stratified splits into training and test subsets in the ratio $80\%/20\%$. Feature standardization and zero-variance filtering are fitted on the training subset and then applied to the test subset.

\begin{table}[H]
    \centering
    \small
    \caption{Data-set dimensions and algorithmic parameters.}
    \label{tab:svm_data_parameters}
    \setlength{\tabcolsep}{4pt}
    \begin{tabular*}{\textwidth}
    {@{\extracolsep{\fill}}lcccccc@{}}
        \toprule
        Data set
        & Samples
        & Features
        & $\lambda$
        & $c_\sigma$
        & $\epsilon_0$
        & $\mu_0$ \\
        \midrule
        WDBC
        & 569
        & 30
        & $7.5\times10^{-4}$
        & 0.18
        & $3\times10^{-2}$
        & 0.30 \\
        Ionosphere
        & 351
        & 34
        & $1.0\times10^{-5}$
        & 0.08
        & $1\times10^{-4}$
        & 1.25 \\
        Spambase
        & 4601
        & 57
        & $3.0\times10^{-5}$
        & 0.12
        & $1\times10^{-4}$
        & 1.25 \\
        QSAR
        & 1055
        & 41
        & $1.0\times10^{-3}$
        & 0.50
        & $1\times10^{-4}$
        & 0.20 \\
        Phishing
        & 11055
        & 30
        & $7.5\times10^{-5}$
        & 0.08
        & $1\times10^{-3}$
        & 1.50 \\
        \bottomrule
    \end{tabular*}
\end{table}

The values of $\lambda$, $c_\sigma$, $\epsilon_0$, and $\mu_0$
in Table~\ref{tab:svm_data_parameters} are fixed before the
experiments and kept unchanged across the ten splits.
For each data set, all six methods use the same $\lambda$.
For the pADMM-based methods, we set
\[
    \sigma=\frac{c_\sigma}{m},
    \qquad
    \eta=
    \frac{(1+\epsilon_0)\|A\|^2}
         {\min\{1,2-\rho\}^2},
    \qquad
    \tau=\frac{1}{\sigma\eta},
\]
where $m$ and $A$ correspond to the preprocessed training
subset in each split. We use $\rho=1.1$ for RH-pADMM and
RFKM-pADMM, and $\alpha=2.2$ for FKM-pADMM
and RFKM-pADMM. For ANPG, $\mu_0$ is the initial smoothing
parameter.

For the pADMM-based methods, we set $p^0=0$. To facilitate progressive support identification, we adopt a strategy similar to that in \cite{TaoBianQu2026}, where all methods start from $x^0=0$ and the cap parameter is gradually decreased according to a common continuation scheme. Specifically, for each training subset, we set $\nu_{\mathrm{final}}=0.99\lambda/(L_f\|A\|)$, where $L_f=1/\sqrt{m}$, which ensures that $\lambda/\nu_{\mathrm{final}}>L_f\|A\|$. The initial cap parameter is chosen as $\nu_0=\min\{18,\max\{10\nu_{\mathrm{final}},5\lambda/\|A^\top\xi^0\|_\infty\}\}$, where $\xi^0=\Pi_{[-1/m,0]^m}(-\sigma\mathbf{1})$. All methods then use the common continuation rule
\[
    \nu_k
    =
    \max\left\{
        \nu_{\mathrm{final}},
        \nu_0(0.75)^k
    \right\},
    \qquad k=1,2,\ldots.
\]
The index $k$ is defined as in the preceding subsection. For the classification experiment, the restart threshold is
$\delta=(\lambda-\nu_{\mathrm{final}}L_f\|A\|)/2$.

Using the notation $z_r^k$ in \eqref{eq:z_r^k}, where $r$ denotes the $r$th data split, we define
\[
    \Delta_{r,k} :=
    \max\left\{
        \frac{\|z_r^k-z_r^{k-1}\|}
             {1+\|z_r^{k-1}\|},
        \frac{
            \left|F_{\nu_{\mathrm{final}}}(z_r^k)
            -F_{\nu_{\mathrm{final}}}(z_r^{k-1})\right|
        }{
            1+\left|F_{\nu_{\mathrm{final}}}(z_r^{k-1})\right|
        }
    \right\},
\]
where $F_{\nu_{\mathrm{final}}}$ denotes the objective function in
\eqref{eq:hinge-capped} with $\nu=\nu_{\mathrm{final}}$. The algorithm
terminates when $\Delta_{r,k}\leq 2\times10^{-3}$ for five consecutive
updates. The maximum number of updates is set to $3000$ for all
methods. For each method, let $k_{\mathrm{stop},r}$ denote the value of $k$
at termination on the $r$th data split. The average number of
nonzero coefficients at termination is defined by
\[
  \mathrm{nnz}
  :=
  \frac{1}{10}
  \sum_{r=1}^{10}
  \left|
  \left\{
  i:
  \left|
  \left(z_r^{k_{\mathrm{stop},r}}\right)_i
  \right|>10^{-8}
  \right\}
  \right|.
\]
\begin{table}[htbp]
    \centering
    \small
    \caption{Average performance over ten data splits.}
    \label{tab:svm_results}
    \setlength{\tabcolsep}{8pt}
    \renewcommand{\arraystretch}{0.92}
    \begin{tabular*}{\textwidth}
    {@{\extracolsep{\fill}}llccc@{}}
        \toprule
        Data set & Method & Test accuracy & Time (s) & $\mathrm{nnz}$ \\
        \midrule

        \multirow{6}{*}{WDBC}
        & FKM-pADMM  & \textbf{0.9816} & $7.04\times10^{-3}$ & 17.1 \\
        & RFKM-pADMM & 0.9807 & $7.09\times10^{-3}$ & 17.1 \\
        & H-pADMM    & \textbf{0.9816} & $6.51\times10^{-3}$ & 17.1 \\
        & RH-pADMM   & 0.9807 & $6.85\times10^{-3}$ & 17.2 \\
        & pADMM      & 0.9763 & $\mathbf{5.61\times10^{-3}}$ & 17.4 \\
        & ANPG       & 0.9781 & $6.76\times10^{-3}$ & 17.5 \\
        \midrule

        \multirow{6}{*}{Ionosphere}
        & FKM-pADMM  & \textbf{0.8944} & $8.00\times10^{-3}$ & 15.5 \\
        & RFKM-pADMM & 0.8859 & $7.25\times10^{-3}$ & 14.6 \\
        & H-pADMM    & 0.8930 & $7.89\times10^{-3}$ & 15.5 \\
        & RH-pADMM   & 0.8859 & $7.43\times10^{-3}$ & 14.6 \\
        & pADMM      & 0.8873 & $8.14\times10^{-3}$ & 16.3 \\
        & ANPG       & 0.8915 & $\mathbf{6.04\times10^{-3}}$ & 14.8 \\
        \midrule

        \multirow{6}{*}{Spambase}
        & FKM-pADMM  & 0.9185 & $\mathbf{1.99\times10^{-2}}$ & 32.9 \\
        & RFKM-pADMM & \textbf{0.9191} & $2.05\times10^{-2}$ & 33.2 \\
        & H-pADMM    & 0.9186 & $2.00\times10^{-2}$ & 32.9 \\
        & RH-pADMM   & \textbf{0.9191} & $2.10\times10^{-2}$ & 33.2 \\
        & pADMM      & 0.9186 & $2.04\times10^{-2}$ & 35.2 \\
        & ANPG       & 0.9181 & $2.68\times10^{-2}$ & 32.7 \\
        \midrule

        \multirow{6}{*}{QSAR}
        & FKM-pADMM  & 0.7853 & $8.59\times10^{-3}$ & 16.5 \\
        & RFKM-pADMM & 0.7853 & $8.38\times10^{-3}$ & 17.2 \\
        & H-pADMM    & 0.7848 & $8.43\times10^{-3}$ & 16.5 \\
        & RH-pADMM   & 0.7848 & $7.84\times10^{-3}$ & 17.2 \\
        & pADMM      & \textbf{0.7929} & $\mathbf{7.77\times10^{-3}}$ & 17.5 \\
        & ANPG       & 0.7863 & $8.58\times10^{-3}$ & 16.9 \\
        \midrule

        \multirow{6}{*}{Phishing}
        & FKM-pADMM  & 0.9152 & $1.56\times10^{-2}$ & 8.0 \\
        & RFKM-pADMM & 0.9150 & $1.58\times10^{-2}$ & 8.0 \\
        & H-pADMM    & 0.9152 & $\mathbf{1.54\times10^{-2}}$ & 8.0 \\
        & RH-pADMM   & 0.9151 & $1.67\times10^{-2}$ & 8.0 \\
        & pADMM      & 0.9142 & $1.80\times10^{-2}$ & 8.0 \\
        & ANPG       & \textbf{0.9153} & $1.65\times10^{-2}$ & 8.0 \\
        \bottomrule
    \end{tabular*}
\end{table}

Table~\ref{tab:svm_results} reports the average test accuracy, computation
time, and nnz of the six methods on the
five data sets.
The results show that the proposed methods remain competitive with
pADMM and ANPG in terms of classification accuracy across the five
data sets. FKM-pADMM and H-pADMM attain the highest average test accuracy
on WDBC, while FKM-pADMM also achieves the highest accuracy on
Ionosphere. On Spambase, the highest average test accuracy is attained by
RH-pADMM and RFKM-pADMM, whereas on Phishing, FKM-pADMM and H-pADMM
achieve accuracies very close to that of ANPG. The main exception is
QSAR, where pADMM attains the highest average test accuracy.

In terms of computation time, FKM-pADMM and H-pADMM are faster than
both pADMM and ANPG on Spambase and are also faster than ANPG on
Phishing while maintaining comparable accuracy. By contrast, pADMM is
the fastest method on WDBC and QSAR, whereas ANPG has the shortest
computation time on Ionosphere. Overall, FKM-pADMM and H-pADMM provide
competitive classification accuracy together with favorable
computational performance on several data sets. Moreover, the
$\mathrm{nnz}$ values obtained by the six methods are similar on each
data set, so their test accuracies are compared at approximately
matched sparsity levels.

To summarize the classification accuracy across data sets, let
$a_{D,j}$ denote the average test accuracy of method $j$ on data set
$D$. We define the cumulative relative accuracy of method $j$ by
\[
    S_j
    :=
    \sum_D
    \frac{a_{D,j}}
         {\max_{\ell}a_{D,\ell}}.
\]
Thus, the best-performing method on each data set has a relative
accuracy of $1$, and the maximum possible value of $S_j$ is $5$.

\begin{figure}[htbp]
    \centering
    \includegraphics[width=0.98\textwidth]
    {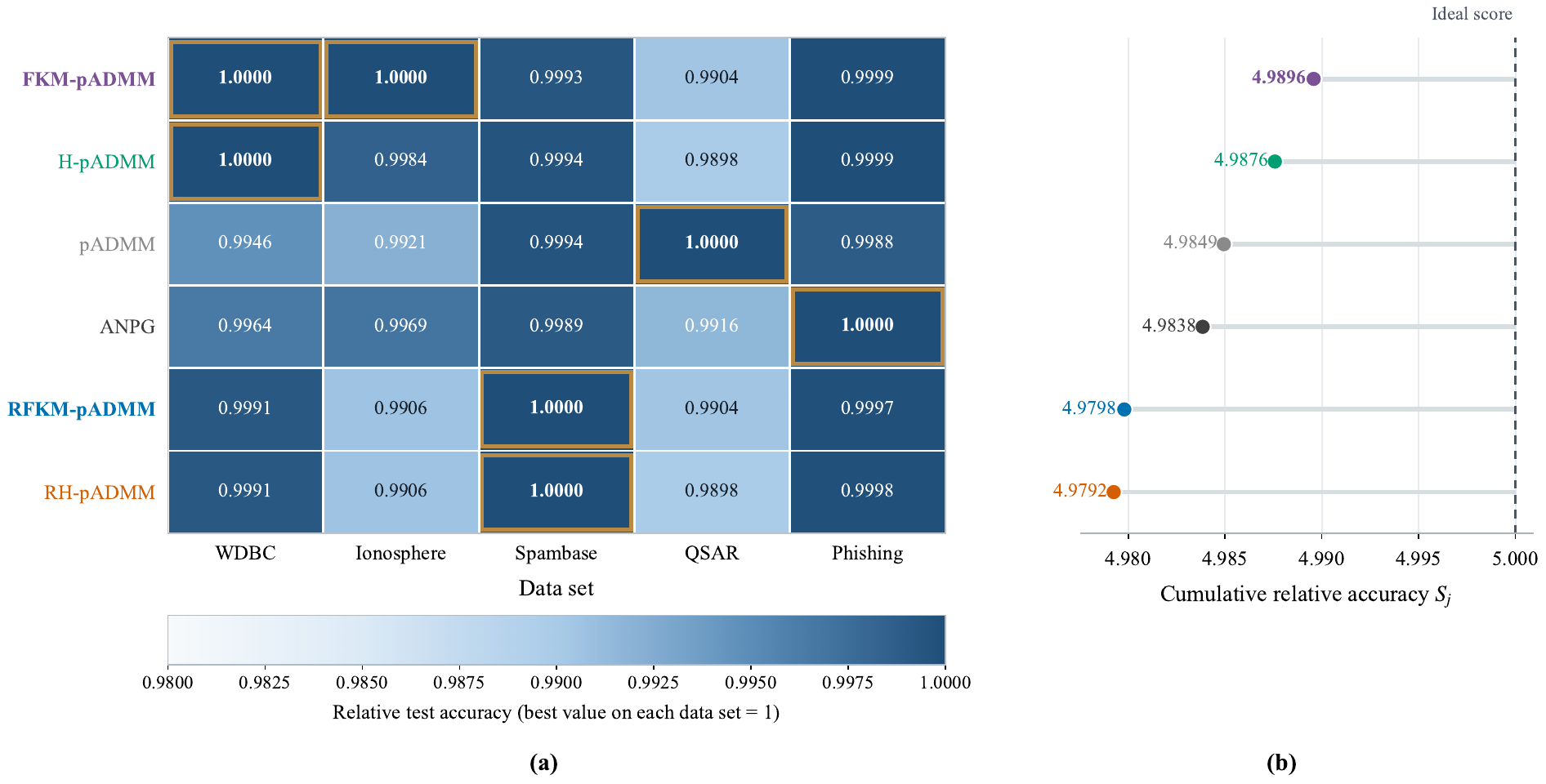}
    \caption{Relative test accuracy on each data set and cumulative
    relative accuracy of the six methods}
    \label{fig:relative_accuracy_comparison}
\end{figure}

The relative test accuracies of the six methods on the five data sets
are compared in Fig.~\ref{fig:relative_accuracy_comparison}(a), while
their cumulative relative accuracy scores are reported in
Fig.~\ref{fig:relative_accuracy_comparison}(b). FKM-pADMM and H-pADMM
achieve slightly higher cumulative relative accuracy than pADMM and
ANPG. In particular, FKM-pADMM attains the highest cumulative relative
accuracy, followed by H-pADMM. These results support the effectiveness
of FKM-pADMM and H-pADMM for sparse classification problems.

\section{Conclusion}\label{sec:conclusion}

In this paper, we developed the RA-pADMM algorithm for solving
problem~\eqref{model} based on the nonexpansive fixed-point
representations of its convex branch problems. We established
finite termination of the preliminary branch identification phase
and finiteness of the branch updates accepted during the
subsequent phase. The primal iterates and proximal outputs
converge to a common lifted stationary point of
problem~\eqref{model}, which is also a $\nu$-strong local
minimizer of problem~\eqref{model_l0}. Moreover, the Halpern and
fast Krasnosel'ski\u{\i}--Mann variants achieve nonergodic
$O(k^{-1})$ and asymptotic $o(k^{-1})$ convergence rates,
respectively, for both the objective value gap and the KKT
residual associated with the final branch. Numerical experiments
on sparse recovery and real data classification problems confirm
the effectiveness of the RA-pADMM algorithm.

\begin{acknowledgements}
The third author was partially supported by the National Natural Science Foundation of China grant (12425115).
\end{acknowledgements}

\par\addvspace{11pt}
{\small
\noindent\textbf{Data Availability}
The datasets generated during the current study are available from the corresponding author on reasonable request.
\par
}

\section*{Declarations}

{\small
\noindent\textbf{Conflict of interest}
The authors declare that they have no conflict of interest.
\par
}

\bibliographystyle{spmpsci}
\bibliography{references}

\end{document}